\documentclass{article}

\usepackage{graphicx}
\usepackage{xxcolor}
\usepackage[colorlinks=true, linkcolor=blue, urlcolor=blue, citecolor   = green]{hyperref}
\usepackage{hyperref}
\usepackage{epstopdf}
\usepackage{amsmath, amssymb}
\usepackage{amsthm}
\usepackage{xcolor}
\usepackage{fancyhdr, blindtext}
\usepackage{tabularx, caption}
\usepackage{multirow}
\usepackage{t1enc}
\usepackage{multicol}
\usepackage[ msc-links]{amsrefs}
\usepackage{tikz}
\usepackage{float}
\usepackage{amsrefs}
\usepackage{python}
\usetikzlibrary{shapes,arrows}
\usepackage{geometry}
\newtheorem{thm}{Theorem}

\theoremstyle{definition}
\newtheorem{exa}{Example}
\newtheorem{prob}{Problem}

\title{Keyed hash function from large girth expander graphs}
\author{Eustrat Zhupa  \\
	University of Rochester \\
  500 Joseph C. Wilson Blvd.\\
Rochester, NY 14627, USA  \\
	\and 
	Monika K. Polak \\
	Rochester Institute of Technology \\
	102 Lomb Memorial Dr\\
	Rochester, NY 14623, USA
	}
\begin{document}

\maketitle
\begin{abstract}
In this paper we present an algorithm to compute keyed hash function (message authentication code MAC). 
Our approach uses a family of expander graphs of large girth denoted $D(n,q)$, where $n$ is a natural number bigger 
than one and $q$ is a prime power.
Expander graphs are known to have excellent expansion properties and thus they also have very good mixing properties.  
All requirements for a good MAC are satisfied in our method and a discussion about collisions and preimage resistance is also part of this work.
The outputs closely approximate the uniform distribution and the results we get are indistinguishable from random sequences of bits. 
Exact formulas for timing are given in term of number of operations per bit of input. Based on the tests,
our method for implementing DMAC shows good efficiency in comparison to other techniques. 
4 operations per bit of input can be achieved. The algorithm is very flexible and it works with messages of any length.
Many existing algorithms output a fixed length tag, while our constructions allow 
generation of an arbitrary length output, which is a big advantage.
\end{abstract}

\section{Introduction}
\label{intro}
Our work is motivated by the work of Charles, Goren and Lauter, \cite{cglhash}.  They proposed the construction of 
collision resistant hash function from expander graphs. The family of graphs they used were Ramanujan graphs constructed by Lubotzky, 
Philips and Sarnak (see \cite{lps88}) and Pizer's Ramanujan graphs, \cite{pizer}.  Hash functions from LPS 
require 7 field multiplications
per bit of input, but the field size may need to be bigger (1024 bit prime p instead of 256 bits, for example), and the output 
is $4 log(p)$ bits. When the hash function from Pizer's graph $G(p,l)$, for $l = 2$, requires  $2 log(p)$ field multiplications
per bit of input, which is quite inefficient (the authors propose to use a graph of cryptographic size $p\approx 2^{256}$). The output 
of this hash function is $log(p)$ bits.
The idea presented in \cite{cglhash} is very good, but collisions for this hash were found (see \cite{Tillich2008}).
However, expander graphs can be used to produce keyed hash function (message authentication code). 
When secret parameters are involved (like colouring and initial vertex) the adversary cannot find collision with mentioned method.

We propose a construction of message authentication code based on another family of expander graphs $D(n,q)$ of large girth, 
\cite{lazebnik1995}.
The most important advantages are: the output can have arbitrary length (like in case of variants of SHA-3: SHAKE128 and SHAKE256),
max message size is unlimited and the performance is very good (4 field operations per bit of input can be achieved). 
Another advantage of our construction is that graphs from the family $D(n,q)$ have a nice representation by vectors and 
incidence relations are described by system of multivariate equations, which is very easy to implement. 
From the other side such systems of nonlinear equations are used for multivariate cryptography that is considered to be a good
candidate for a post-quantum cryptography, \cite{Goubin2011}.

The basics about hash function and keyed hash function (message authentication code; MAC) can be found in \cite{stallings}.
A hash function  accepts a message $M$  as input and produces a fixed-size hash value $h = H(M)$.
Hash functions are often used to determine whether or not data has changed. In general terms, the main goal of a
hash function is to ensure data integrity.However it can be also used for authentication, to create one way password files, 
as a source of pseudorandom numbers (bits) or for intrusion and virus detection.
A cryptographic hash function is a function that is acceptable for security applications.
It means that it shall be computationally infeasible to find

\begin{description}
	\item [i] a $x$ that is the preimage of $h$ for a hash value $h = H(x)$ (one way property),
	\item [ii]  two data objects  $x \neq y$, for which $H(x) = H(y)$ (the collision-free property). 
\end{description}
In addition, a good hash function has the property that the output looks like random data and even a small change in 
input causes big changes in the output.

It is possible to use a hash function but no encryption for a message authentication.
There are a few techniques to achieve this: use hash function + encryption on the hash, compute a hash value over the 
concatenation of $M$ and $S$ (a common secret value) and append the resulting hash value to the message or use keyed hash function.
There are many reasons why it is worth to use techniques that
avoid encryption \cite{TSUD92}. For example, if there is no need to keep message confident but we want to authenticate it, those 
techniques are faster.

Figure ~\ref{mac} illustrates the mechanism for message authentication using keyed hash function.
In order to check the integrity of the message, the keyed hash function is applied to the message and the result 
is compared with the associated tag. The secret key $K_s$ is known for the receiver and for the sender, so 
the sender can be easily verified. The MAC approach guarantees data integrity and authenticity.

\begin{figure}
\centering
\includegraphics[width=3.5in]{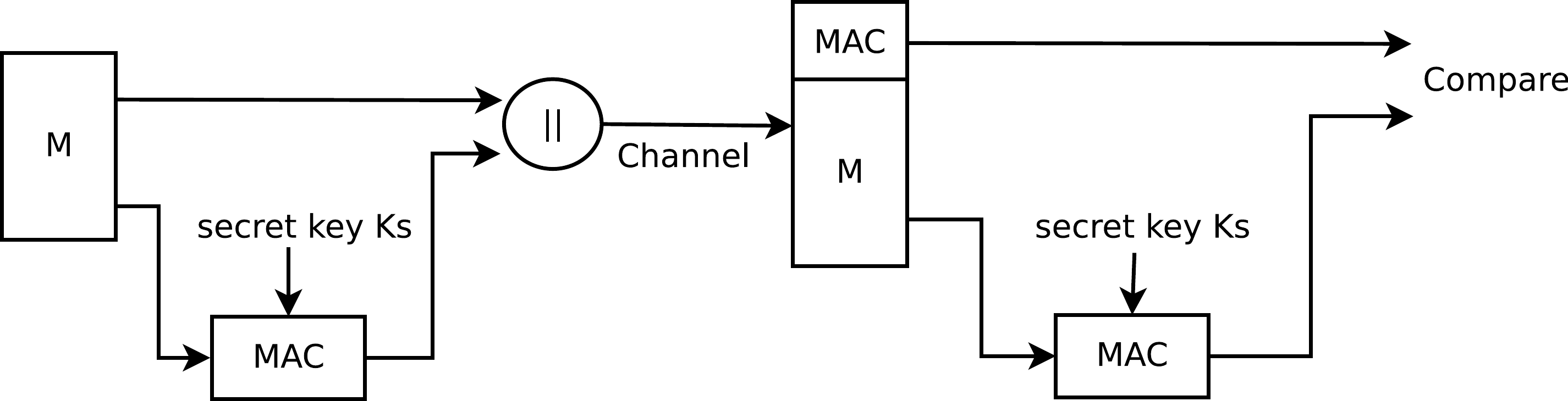}
\caption{Message authentication using MAC}
\label{mac}
\end{figure}

\section{Background}
\label{sec:2}
\paragraph{Graphs of large girth}
We define the \emph{girth} of a graph as the length of the shortest cycle. In the analyzed context girth is a very important property 
of the graph. 
Let $\{G_i\}_{i\in\mathbb N}$ be a family of $k$-regular graphs with increasing order. Let  $g_i$ and $v_i$ denote respectively 
the girth and the 
order of the graph $G_i$. A family of graphs with increasing girth is a sequence of graphs such that $g_i\leq g_j$ for $i< j$.
According to the definition introduced in \cite{biggs88}, we say that a family of $k$-regular graphs is a family of graphs of large girth if
\[g_i \geq \gamma \log_{k-1}(v_i)\]
for a constant $\gamma$ and all $i \geq1$.

A nice survey about graphs of large girth is presented in \cite{biggs98}.
It is known that $\gamma = 2$ (see \cite{boll2004}) is the best possible constant, but there is no explicit construction of such 
family of graphs for which it can be obtained. This topic started 
in 1959 when Paul Erd\H{o}s proved existence of such families with bounded degree $k$ and $\gamma=1/4$, without providing a 
construction \cite{erd59}. 
There have been numerous investigations of the field by several authors. However, until 2017 the list of major known results is short. 
The list of explicit constructions is the following:

\begin{enumerate}
	\item the first explicit construction of such family with \mbox{$\gamma=0.44$}, denoted by $X(p,q)$, where $p$ and $q$ are 
	primes was introduced by  G. A. Margulis in 1982 \cite{marguli82},
	\item generalisation of the family $X(p,q)$ proposed by M. Morgenstern \cite{morg},
	\item constructions for arbitrary $k$ with $\gamma=0.48$ and construction of family of 3-regular graphs with $\gamma=0.96$
	obtained by  V. Imrich in 1983 \cite{imrich84},
	\item family of sextet graphs introduced in 1983 by Biggs and Hoare \cite{biggshoare83} (Alfred Weiss \cite{aweiss1984} 
	proved that $\gamma=\frac{4}{3}$),
	\item second construction by G. A. Margulis in 1988 \cite{marguli88},
	\item constructions of cubic graphs, presented in a popular article \cite{biggs98},
	\item construction by Lubotzky, Phillips and Sarnak \cite{lps88} (Biggs and Boshier \cite{biggsboshier90} proved that $\gamma=\frac{4}{3}$ for 
	this family of graphs),
	\item algebraic graphs $CD(n,q)$ given by the nonlinear system of equations over finite field
$\mathbb{F}_q$, with \mbox{$\gamma\geq\log_q(q-1)$}, \cite{usty1995} (Furedi \cite{furedi1995} proved that for arbitrary prime 
power $q$: $\gamma=log_q(q-1)$),
\item the polarity graphs of $CD(n,q)$ (see \cite{laz99}) has an induced subgraph of degree $q-1$ which is a family of graphs of large 
girth (it is shown in \cite{usty2005}).
\end{enumerate}
\paragraph{Expander graphs}
An important property that a family of graphs must have in order to be a good candidate for a construction of a hash function is to be
a family of expander graphs. Cryptographic hash function from expander graphs were presented in works of \cites{zem1,zem2,cglhash}. 

Let's consider a spectrum of a graph with eigenvalues  $\lambda_0 > \lambda_1 > . . . > \lambda_{s-1}$.
A family of $k$-regular graphs of increasing order $\{G_i\}_{i\in\mathbb N}$ is called a family of Ramanujan 
graphs if $|\lambda_1(G_i)| \leq 2\sqrt{k - 1}$ for all $i$ (\cite{expander}, p. 452). 
If $k$ is stable 
and $v_i\rightarrow\infty$ the limit $2\sqrt{k - 1}$ is the best we can get. Ramanujan graphs are the best expanders. 

\section{The family of graphs $D(n,q)$}
\label{sec:4}
The family of graphs we use in our construction of a MAC was introduced in 1992 by Lazebnik and Ustimenko, \cite{usty1995}. 
It is denoted by $D(n,q)$, $n\in \mathbb N_{\geq 2}$ and $q$ is a prime power. The similar notations that we use appeared later and can be found in \cite{lazebnik1995}. 
To simplify, we don't use double notations for coordinates of vectors.
The family $D(n,q)$ is special because of a few important properties. The first one is that this is a 
family of graphs of a large girth, as mentioned in Sec.~\ref{sec:2}. Secondly, this is a family of very good 
expander graphs that are close to Ramanujan graphs. The idea of almost Ramanujan graphs was introduced in \cite{ustimenko1998}. 
We refer to a family of $k$-regular graphs as almost Ramanujan graphs if $|\lambda_1(G_i)| \leq 2\sqrt{k}$ for all $i$.  
Graphs  $D(n, q)$, $n \geq 2$ for arbitrary $q$ form a family of \mbox{$q$-regular} almost Ramanujan graphs ($|\lambda_1(G_i)| \leq 2\sqrt{q}$) 
and thus have excellent mixing properties.

Graphs $D(n,q)$ are bipartite with set of vertices $V$ containing two subsets: $V=P\cup L$, 
where $P\cap L=\emptyset$. Originally the subset of vertices $P$ is called a set of points and another set $L$ is called a set of lines. 
Let $P$ and $L$ be two copies of Cartesian power ${\mathbb{F}_q}^n$, where $n\geq 2$ is a integer. 
Two types of brackets are used in order to distinguish points from lines.
We write $(\vec{z})$ if $\vec{z}\in P$ and $[\vec{z}]$ if $\vec{z}\in L$.
The set of vertices of graph $D(n,q)$ (collection of points and lines) can be considered as  $n$-dimensional vectors over $\mathbb{F}_q$:
\begin{align*}
(\vec{p}) =& (p_{1}, p_{2}, p_{3}, p_{4}, . . .,p_n),\\
[\vec{l}] =& [l_{1}, l_{2}, l_{3}, l_{4},. . . , l_n].
\end{align*}

Coordinates of $(\vec{ p})$ and $[\vec{ l}]$ are elements of finite field $\mathbb{F}_q$. Because of this we 
have: $|P|=|L|=q^n$ and $|V|=2q^n$.
The vertex $(\vec{p})$ (point $(\vec{p})$) is incident with the vertex $[\vec{l}]$ (line $[\vec{l}]$) and we write: $(\vec{p})I[\vec{l}]$, 
if the following
relations between their coordinates hold:
\begin{equation}\label{1}
\left\{
  \begin{array}{ll}
  l_{2}-p_{2} = l_{1}p_{1}\\
  l_{3}-p_{3} = l_{2}p_{1}\\
  l_{4}-p_{4} = l_{1}p_{2}\\
	l_{i}-p_{i} = l_{1}p_{i-2}\\
  l_{i+1}-p_{i+1} = l_{i-1}p_{1}\\
  l_{i+2}-p_{i+2} = l_{i}p_{1}\\
  l_{i+3}-p_{i+3} = l_{1}p_{i+1}
  \end{array}
\right.
\end{equation}
where $i\geq5$.
The set of edges $E$ consists of all pairs $((\vec{ p}),[\vec{l}])$ for which $(\vec{ p})I[\vec{ l}]$. 
This is a family of $q$-regular graphs, which means that each vertex has exactly $q$ neighbors.
 $D(n,q)$ becomes disconnected for
$n \ge 6$. Graphs $D(n, q)$ are edge transitive. It means that their connected components are isomorphic. 
A connected component of $D(n,q)$ is denoted by $CD(n,q)$. Notice that all connected components of infinite graph $D(q)$ are $q$-regular 
trees. The length of the shortest cycle (the girth) of a graph $D(n,q)$ is given by the formula:
\[g(D(n,q))=
\begin{cases}
n+5, \text{ for odd }n\\
n+4, \text{ for even }n
\end{cases}\]

Graphs $D(n,q)$ were successfully used as a base for symmetric and public key multivariate cryptography (see for example: 
\cites{ustimenko2013,klisowski2012,muav2013,usty2005,Ustimenko2007}),  error correcting codes (see \cite{gl1997}) and pseudorandom 
number generator \cite{generator2017}. The related cryptosystems are very good candidates for post quantum cryptography and 
resistant to linearization attacks. The base of our message authentication code is a stream cipher algorithm where 
most of the constructions of message authentication functions are based on block ciphers.

\section{Keyed hash function}
Few notations are used in this work. Let $M$ denote the message, $N$ the number of bits per block of the $M$ and by 
$l(M)$ we represent the number of blocks. So the message can be expressed as $M=m\ldots m_{l(M)}$. We consider that the message is 
written in alphabet that corresponds to elements of finite field $\mathbb F_q$ and by $l(q)$ we denote the number of bits
needed to represent number $q$ (for example UTF-8 uses number field $\mathbb F_{2^8}$). 
Calculations shall be performed in bigger number field than the number field  ($\mathbb F_{q}$) that is used for the alphabet in order to
achieve collection resistance property.
Let denote by $\mathbb F_{Q}$ the number field used for calculations. The choice of $N$ determines $\mathbb F_Q$.
Any change in message shall change the hash so different input blocks $m_i$ shall correspond to different edge colouring.
To achieve it the following condition shall be satisfied \[Q\geq 2^N.\]
It is convenient to choose $Q=p$, where $p$ is a prime number. In such case field arithmetic is simply modulo $p$ arithmetic. 

From now on we denote by $h$ the size of output (tag).
The input message $M$ is used as direction for walking around the graph $D(n,Q)$. We start with initial vertex $IV=\vec{v_0}$, which 
we consider to be a point 
($(\vec{v_0})\in P$). The next visited vertex is obtained by the formula
\[N_t(\vec{w}=(w_1,w_2,\ldots,w_n))=[(w_1+t)^2,\underbrace{\ast,\ldots,\ast}_{n-1}],\]
\[N_t(\vec{w}=[w_1,w_2,\ldots,w_n])=((w_1+t)^2,\underbrace{\ast,\ldots,\ast}_{n-1}),\]
where $\ast$ can be uniquely calculated from equations~(\ref{1}) (see Example~\ref{ex:1}). Recall that $D(n,q)$ graphs are 
bipartite: points cannot be incident to points and lines cannot be incident to lines.
\begin{exa}\label{ex:1}
Let consider graph $D(6,11)$ and $\vec{w}\in P$.
\[N_3(\vec{w}=(1,8,4,2,7,0))=[(1+3)^2,\ast,\ast,\ast,\ast,\ast]\]
Names are assigned for $\ast$: $[(1+3)^2,\ast,\ast,\ast,\ast,\ast]=[5,l_2,l_3,l_4,l_5,l_6]$. Then
\begin{equation}
\left\{
  \begin{array}{ll}
  l_{2}-8 = 5\cdot1\\
  l_{3}-4 = l_{2}\cdot1\\
  l_{4}-2 = 5\cdot8\\
  l_{5}-7 = 5\cdot4\\
  l_{6}-0 = l_{4}\cdot1\\
  \end{array}
\right.
\end{equation}
where all operations are in finite field $\mathbb F_{11}$. The calculated neighbor of $(\vec{w})$ is $[5,2,6,9,5,9]$. 
\end{exa}

We propose two approaches to calculate the keyed hash function based on this family of graphs. We named the message 
authentication codes DMAC, because constructions are based on family of graphs $D(n,q)$. Keyed hash functions use a secret, which is used
to calculate the hash. We propose a secret key to be a pair $(IV,S)$. $IV$ is an initial vector of length $n$ with coordinates from $\mathbb F_q$. 
$S$ is a password of $s$ characters from alphabet $\mathbb F_q$ such that
\[s\leq \frac{1}{2}g(D(n,Q)).\]
In our constructions, after all blocks $m_i$ of a message are processed, we process a password $S$.
The details are described in the next subsections.
 
\subsection{Basic construction (DMAC-1)}
Fig.~\ref{proc_messg} illustrates the first type of proposed DMAC's.
\begin{figure}
\centering
\includegraphics[width=4.5in]{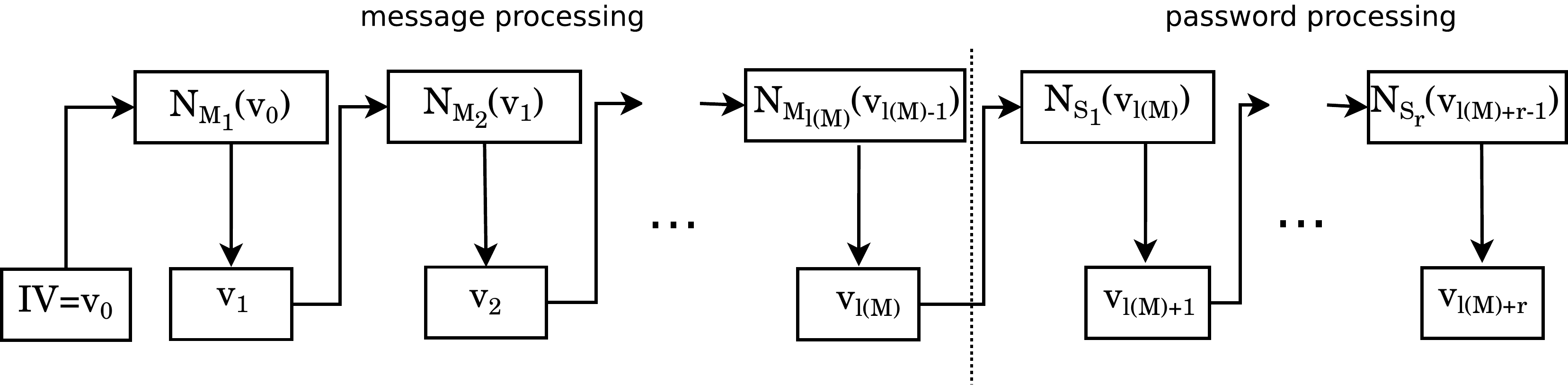}%
\caption{General structure of the DMAC-1}
\label{proc_messg}
\end{figure}
Steps to authenticate the message with DMAC-1:
\begin{enumerate}
	\item Agree secret key $K_s$, which is a pair $(IV,S)$.
	\item split $M$ in blocks $m_i$ of length $N$ (add padding if needed) 
	\item Process the message. For $i=0,\ldots,l(M)-1$ do
	\begin{itemize}
		\item Characters of the block of message  $m_i$ are concatenated to obtain a number $M_i$.
		\item Calculate the vertex $\vec{v_{i+1}}$ which is incident to vertex $\vec{v_i}$. So, we calculate the next 
		visited vertex by using operator $N_{M_i}(v_i)$:

\[N_{M_i}(\vec{v_i}=(v_1^i,v_2^i,\ldots,v_n^i))=[(v_{i\texttt{mod} n+1}^i+M_i)^2\mod Q,\underbrace{\ast,\ldots,\ast}_{n-1}],\]
\[N_{M_i}(\vec{v_i}=[v_1^i,v_2^i,\ldots,v_n^i])=((v_{i\texttt{mod} n+1}^i+M_i)^2\mod Q,\underbrace{\ast,\ldots,\ast}_{n-1}),\]

      where $\ast$ are calculated from equations~(\ref{1}). We start in vertex $\vec{v_0}$ that is equal to $IV$.	
      \end{itemize}
\item Process the password $S$. For a $i=l(M),\ldots,l(M)+r-1$ do
	\begin{itemize}
		\item Calculate the vertex $\vec{v_{i+1}}$ which is incident to vertex $\vec{v_i}$. 
		So, we calculate the next visited vertex by using operator $N_{S_i}(\vec{v_i})$:
\[N_{S_i}(\vec{v_i}=(v_1^i,v_2^i,\ldots,v_n^i))=[(v_{i\texttt{mod} n+1}^i+S_i)^2\mod Q,\underbrace{\ast,\ldots,\ast}_{n-1}],\]
\[N_{S_i}(\vec{v_i}=[v_1^i,v_2^i,\ldots,v_n^i])=((v_{i\texttt{mod} n+1}^i+S_i)^2\mod Q,\underbrace{\ast,\ldots,\ast}_{n-1}),\]
   where $\ast$ are calculated from equations~(\ref{1}). We start in vertex $\vec{v}_{l(M)}$ (the last visited vertex in step 3).	
\end{itemize}
\end{enumerate}

\subsection{Modified construction (DMAC-2)}
For $n\geq 6$ graphs become disconnected. In order to move from one component to another we can use simple 
modifications presented in Fig.~\ref{proc_messg_2}. A vectors addition $\fbox{+}$ over $\mathbb F_Q$ is added.

\begin{figure}
\centering
\includegraphics[width=4.5in]{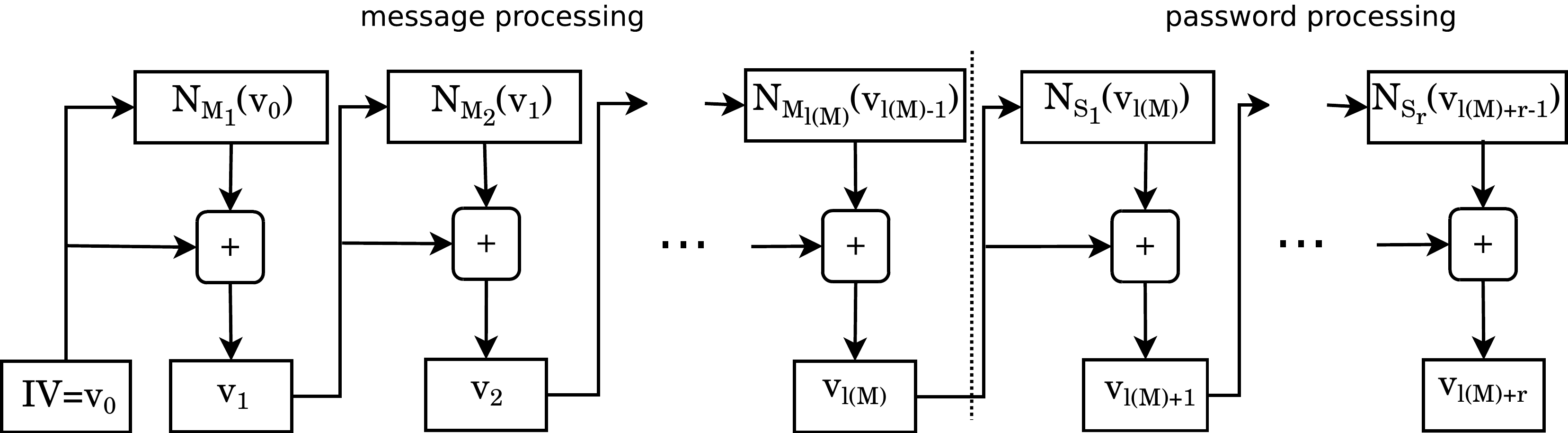}
\caption{General structure of the DMAC-2}
\label{proc_messg_2}
\end{figure}

In this case steps to authenticate the message with message authentication 
code are the same like for DMAC-1 except one additional step.
Steps to authenticate the message with DMAC-2:
\begin{enumerate}
	\item Agree secret key $K_s$, which is a pair $(IV,S)$.
	\item split $M$ in blocks $m_i$ of length $N$ (add padding if needed) 
	\item Process the message. For $i=0,\ldots,l(M)-1$ do
	\begin{itemize}
		\item Characters of the block of message  $m_i$ are concatenated to obtain the number $M_i$.
		\item Calculate the vertex $\vec{v_{i+1}}$ which is incident to vertex $\vec{v_i}$. 
		So, we calculate the next visited vertex by using operator $N_{M_i}(\vec{v_i})$. We start in 
		vertex $\vec{v_0}$ that is equal to $IV$.
	\item  Add vectors $\vec{v_i}$ and $\vec{v_{i+1}}$ over $\mathbb F_Q$.
\end{itemize}
		
	\item Process the password $S$. For a $i=l(M),\ldots,l(M)+r-1$ do
	\begin{itemize}
		\item Calculate the vertex $\vec{v_{i+1}}$ which is incident to vertex $\vec{v_i}$. 
		So, we calculate the next visited vertex by using operator $N_{S_i}(\vec{v_i})$. 
		We start in vertex $\vec{v}_{l(M)}$ (the last visited vertex in step 3).\item  Add vectors $\vec{v_i}$ 
		and $\vec{v_{i+1}}$ over $\mathbb F_Q$.
\end{itemize}
\end{enumerate}

\begin{exa} (A toy example)
Let's consider the following example. The alphabet is $\mathbb F_{29}$. We want to calculate DMAC-2 of 15 bits  ($h=15$) for a 
message $M$ and a secret key $K_s$.
\begin{center}
\begin{tabular}{c|c|c|c|c|c|c|c|c|c|c}
\hline\noalign{\smallskip}
A&B&C&D&E&F&...&Z&.&,&\textendash \\
0&1&2&3&4&5&...&25&26&27&28\\
\noalign{\smallskip}\hline
\end{tabular}
\end{center}
$M:$A\textendash BEAUTIFUL\textendash DAY corresponds to the vector $(0,28,1,4,0,...,24)$\\
$S:$.AY corresponds to the vector $(26,0,24)$\\
$IV=(5,10,27)=\vec{v_0}$

In this case $l(q)=5$ bits. We shall use $n$ that satisfies: $h\leqslant n\cdot l(q)$. Hence $n=3$.
If we set $N=25$ bits then each block has 5 characters and we decide to use $Q=33554467$. We use graph $D(3,33554467)$.
 We have 3 blocks ($l(M)=3$) total and the last block is padded:
\[m_1=(0,28,1,4,0)\]
\[m_2=(20,19,8,5,20)\]
\[m_3=(11,28,30,24,0)\]
\begin{enumerate}
 \item $i=0$ 
 \[\begin{split}
  \vec{v_1}&=N_{M_1}(\vec(v_0)=[(v_1^0+M_1)^2\mod33554467,v_2^1,v_3^1]\\
 &=[(5+28140)^2\mod 33554467,v_2^1,v_3^1]\\
 &=[20388284,1278029,6390172]
 \end{split}\]
 because
 \begin{equation*}
\left\{
  \begin{array}{ll}
  v_{2}^1-10 = 20388284\cdot5\\
  v_{3}^1-27 = v_2^1\cdot5\\
  \end{array}
\right.
\end{equation*}
\[\vec{v_1}:=\vec{v_1}+\vec{v_0}=[20388289,1278039,6390199]\]
 \item $i=1$ 
 
 \[\begin{split}
  \vec{v_2}&=N_{M_2}(\vec(v_1)=((v_2^1+M_2)^2\mod33554467,v_2^2,v_3^2)\\
 &=((1278039+20198520)^2\mod 33554467,v_2^2,v_3^2)\\
 &=(30968786,21891813,29421730)
 \end{split}\]
 because
 \begin{equation*}
\left\{
  \begin{array}{ll}
  1278039-v_{2}^2 = 20388289\cdot30968786\\
  6390199-v_{3}^2 = 1278039\cdot30968786\\
  \end{array}
\right.
\end{equation*}
\[\vec{v_2}:=\vec{v_2}+\vec{v_1}=(17802608, 23169852, 2257462)\]
 \item $i=2$ 
 \[\begin{split}
  \vec{v_3}&=N_{M_3}(\vec(v_2)=[(v_3^2+M_3)^2\mod33554467,v_2^3,v_3^3]\\
 &=[(2257462+112830240)^2\mod 33554467,v_2^3,v_3^3]\\
 &=[14009975,4873348,10691714]  
   \end{split}\]
 because
 \begin{equation*}
\left\{
  \begin{array}{ll}
  v_{2}^3-23169852 = 14009975\cdot17802608\\
  v_{3}^3-2257462 = v_2^3\cdot17802608\\
  \end{array}
\right.
\end{equation*}
\[\vec{v_3}:=\vec{v_3}+\vec{v_2}=[31812583, 28043200, 12949176]\]
 \item $i=3$ $\dots$
 \item $i=4$ $\dots$
 \item $i=5$ $\dots$
\[\vec{v_6}:=(\vec{v_6}+\vec{v_5})\mod 29=h.\]
\end{enumerate}
\end{exa}

\subsection{Properties of DMACs}

A cryptographic hash function must work as follows: a small change in the input  drastically changes the output. 
This is called \emph{avalanche effect}. DMAC-1 and DMAC-2 were implemented and tested in Python. Our DMACs are the case of a 
high-quality keyed hash functions (see Table \ref{tab:random}). Results presented in table are for the following parameters: 
graph $D(32,257)$, $N=32$, $h=256$.  
Output of the presented keyed hash functions passed the well known  Diehard tests, developed by George Marsaglia, for measuring the 
quality of random number generators, \cite{diehard}.

As defined above $h$ is the number of bits of the output (tag). Popular size of tags are $h=128$ bits, $h=256$ bits, $h=512$ bits 
and $h=1024$ bits. 
When for most of commonly used algorithms the size of tag is fixed (for example: SHA-3-224 and SHA-3-256), in our approach the tag can 
have arbitrary length.
The size of block length $N$ can be chosen quite arbitrarily but it has to be much smaller than the size of the 
 message $M$ and $N\geq l(q)$. 
 The longer the size of block, the more efficient the algorithm.
 Notice that if $l(q)=N$ then message is encoded character by character and it becomes a kind of 'string' algorithm.

The parameters of graph  $D(n,Q)$ that is used depend on the block size $N$ and the size of tag $h$. The paramater $Q$ is chosen to satisfy the property the 
inequality $Q\geq 2^N$ and the parameter $n$ is chosen as the  smallest possible $n$ that satisfies the inequality 
\[h\leqslant n\cdot l(q).\]
The most commonly used encodings are UTF-8 ($l(q)=8$), UTF-16 ($l(q)=16$), UTF-32 ($l(q)=32$). 
The Table \ref{tab:n} presents example values of $n$ when $h$ and alphabet $\mathbb F_q$ ($l(q)$) are fixed. 

\begin{exa}\label{ex:2}
Let's consider a message $M$ of 2000 characters writen in UTF-8 (alphabet $\mathbb F_{2^8}$; $l(q)=8$) parameter $S$ of 10 characters . 
We want to divide the message on blocks of 4 characters ($N=32$ bits) and compute a tag of length $h=512$ bits 
for this message.\\
Parameter $n$ can be computed from the formula $h\leq n\cdot l(q)$:
\[512=8n \Rightarrow 2^9=2^3n \Rightarrow n=64.\]
Then we choose a prime power $Q$ such that
\[Q\geq 2^{32} \Rightarrow Q=2^{32}.\]
Thus we use $D(64,2^{32})$ graph. The length of the shortest cycle in this graph is $g(D(64,2^{32}))=68$  and the order of the
graph is $2\cdot (2^{32})^{64}=2^{2048}$.
\end{exa}

\begin{table}

\caption{Example $n$ values for fixed tag size $h$ and coding}
\label{tab:n}
\centering

\begin{tabular}{lllll}
\hline\noalign{\smallskip}
 & & UTF-8 & UTF-16 & UTF-32\\ \cline{3-5}
\noalign{\smallskip}
\multirow{4}{*}{$h$}& $2^7=128$ bits&$2^7=n2^3\Rightarrow n=2^4$&$2^7=n2^4\Rightarrow n=2^3$& $2^7=n2^5\Rightarrow n=2^2$\\
&$2^8=256$ bits&$n=2^5$&$n=2^4$&$n=2^3$\\
&$2^9=512$ bits&$n=2^6$&$n=2^5$&$n=2^4$\\
&$2^{10}=1024$ bits&$n=2^7$&$n=2^6$&$n=2^5$\\
\noalign{\smallskip}\hline
\end{tabular}
\end{table}

\section{Collision resistance and one way property}
\label{sec:col}
Recall that, the family of graphs is a family of graphs of large girth and $g(D(n,q))=2[(n+5)/2]$. Hence there are no cycles shorter than
$2[(n+5)/2]$ and therefore for numer of blocks smaller than $[(n+5)/2]$ no collisions can be find.
First, we consider the following problems.

\begin{prob}
 Find a cycle in graph $D(n,Q)$ that passes through vertex $v_0$ and $v_{l(M)}$.
\end{prob}

\begin{prob}
 Find a path between vertex  $v_0$ and $h$ in graph $D(n,Q)$, that contains a subpath defined by $S$ that ends in $h$.
\end{prob}

First of all we shall notice that the secret key is a pair $(IV=v_0,S)$. For an adversary that doesn't know the secret key those problems
are not defined precisely.
The problem of collision resistance is essentially the problem of finding a shortest cycle in the graph $D(n,q)$ (similarly as 
it was considered in \cite{cglhash} for other graphs). We have the following theorem.
\begin{thm}\label{th:1}
 Finding a collision in DMACs is a solution to Problem 1.
\end{thm}
\begin{proof}
If we set $v_0$ to be zero vector then DMAC-1 and DMAC-2 are exactly the same functions. Finding a collision in DMAC-2 cannot be easier than finding
a collision in DMAC-1. Therefore, without lost of generalisation we can consider collision resistance of DMAC-1.

To compute a hash DMAC-1 we start a walk in vertex $v_0$ and an input $M=m_1m_2\dots m_{l(M)}$ (message $M$) gives us directions $M_i$ how to walk in this graph.
Each vertex is $Q$-regular and $M_i\leq Q$ so different blocks correspond to different edges incydet to a given vertex. 
To find a collision we have to find two different inputs $M\neq M'$, which hash to the same output $h$. 
To calculate output $h$ first we have to calculate vertex 
$v_{l(M)}$ and then using secret $S$ we can walk to  vertex that corresponds to $h$. 
If $v_{l(M)}\neq v_{l(M')}$ then $S$ such that: $s<\frac{1}{2}g(D(n,Q))$ would lead us to different $h$ and $h'$ ($g(D(n,Q)$ denotes the length of the shortest cycle in graph).
Hence, to find a collision we have to find two different inputs $M\neq M'$, which leads us to one vertex $v_{l(M)}=v_{l(M')}$.
Two paths in graph, that start and end in the same vertices form a cycle.
\end{proof}

Graphs $D(n,q)$ form a family of a simple graphs. In this case $O(|V|)$ time is required to find a cycle in an $|V|$-vertex graph,
that starts in a given vertex. However here $|V|=2Q^n$ so the complexity becomes exponential. 
Then $O(|V|)$ is a time required requaired to find any cycle. In our case we are looking for a specific cycle that contains also $v_{l(M)}$.\\
The family of graphs is a family of graphs of large girth and 
$g(D(n,q))=2[(n+5)/2]$. Therefore, the problem of finding a shortest cycle in $D(n,q)$ graphs cannot be easier than the general 
problem of finding the shortest path in a regular graph, which is considered to be hard.
\begin{thm}
 Finding a preimage of $h$ is a solution to Problem 2.
\end{thm}

\begin{proof}
Because of the reason given in the proof of Theorem \ref{th:1}, without lost of generalisation we can consider preimage resistance of DMAC-1.
If we have knowledge about $S$ then vertex $v_{l(M)}$  can be computed.
An input $M=m_1m_2\dots m_{l(M)}$ message gives us directions $M_i$ how to walk in this graph. We start a walk in initial vertex $IV=\vec{v_0}$.
The second visited vertex is defined by the operator $N_{M_1}(\vec{v_0})$ and uniquely determined from equations~(\ref{1}).
The next visited vertex is defined by the operator $N_{M_2}(\vec{v_0})$ and uniquely determined from equations~(\ref{1}). 
We repeat the calculations until we deal with all blocks $m_i$.
Graph $D(n,Q)$ is $Q$-regular and $M_i\leqslant Q$ so different $M_i$ gives us different directions. Each $M_i$ corresponds exactly to one edge incidence to a given vertex. There are many different paths from $v_0$ to $v_{l(M)}$. Find the preimage is 
to find the right path from $v_0$ to $v_{l(M)}$. Notice that $v_0$ is a part of a secret key.
\end{proof}

Composition of operators $N_{M_i}$ and $N_{S_i}$ gives a nonlinear system of $n-1$ cubic eqations (see Theorem 2 in \cite{aw}). Variables are:
numbers $M_i$, $s$ character of $S$ and $n$ coordinates of initial vertex $v_0$. There is $l(M)+n+s$ total variables in this system.\\
In general, solving a set of quadratic equations over a finite field is NP-hard ($MQ$ problem) for any finite field. 
 There is a conjecture that this is a  probabilistically hard problem and Shor’s algorithm cannot be used to speed it up, \cite{Goubin2011}. 
 Solving a set of cubic equations over a finite field cannot be easier than solving the  $MQ$ problem. 
 However, the system related to the set of equations \ref{1} and other systems used for multivariate cryptography are not random, for a large enough parameters it is computationally 
 infeasible to solve them (see \cite{Goubin2011}).

Brute force attack to completely break the keyed hash function (find secret $K_s$) may require to check $q^nq^r$ possibilities ($q^n$ 
possible initial vectors and $q^r$ possible passwords of length $r$), if we consider that the length of the $S$ is known.
A very efficient algorithm to find the shortest path in a graph is Dijkstra's algorithm of complexity $O(|V|\log|V|+|E|)$ and it can be 
adopted to find collisions.
In the case of the used graphs it gives $O(2Q^n\log(2Q^n)+Q^{n+1})$ and it's not more efficient than brute force. 
The complexity is increased because calculations are made over bigger number field $\mathbb  F_{Q}$, without changing the alphabet 
for $IV$ and $S$.
\small{
\begin{table}
\caption{Summary of the Tests}
\label{tab:random}
\centering
\begin{tabular}{p{0.08\linewidth}p{0.2\linewidth}p{0.2\linewidth}p{0.05\linewidth}p{0.37\linewidth}}
\hline\noalign{\smallskip}
	MAC type & Original Message & Output & $S$ & IV\\
\noalign{\smallskip}\hline\noalign{\smallskip}
\multirow{4}{*}{DMAC-1} & The sky announced a beautiful day: the setting moon shane pale in an immense field of azure, which, 
towards the east, mingled itself lightly with the rosy dawn. & 
49 a7 df d0 58 51 6a 9d 4e 94 2d 43 2a b9 60 f2 ab 22 5a a8 18 13 
20 7d f7 1 5f ad 21 3f 56 45 & hint & [147, 217, 2582, 2976, 1718, 1599, 27, 1083, 471, 1461, 1076, 2255, 
 2875, 2696, 2793, 1015, 1477, 1271, 2856, 221, 961, 2839, 1789, 1845, 1157, 622, 758, 882, 210, 1846, 3009, 410]

\\

 & The sky ... & 
2d 48 b6 e4 50 de 8b d2 2e b4 1b d fb 9f b6 63 a1 7b e2 ee 4 e7 b1 ed 88 25 51 ca c4 7d e3 36
 & hint & [\underline{149}, \underline{219}, 2582, 2976, 1718, 1599, 27, 1083, 471, 1461, 1076, 2255, 
 2875, 2696, 2793, 1015, 1477, 1271, 2856, 221, 961, 2839, 1789, 1845, 1157, 622, 758, 882, 210, 1846, 3009, 410]

\\
 & The sky ... & 
78 46 b2 50 81 c1 ba b2 c c4 e6 6c 7a 69 b4 fd c6 64 a 69 30 
e0 4d 30 1e e7 9c 36 55 e d1 8a & h\underline{u}nt & [147, 217, 2582, 2976, 1718, 1599, 27, 1083, 471, 1461, 1076, 2255, 
 2875, 2696, 2793, 1015, 1477, 1271, 2856, 221, 961, 2839, 1789, 1845, 1157, 622, 758, 882, 210, 1846, 3009, 410]

\\
 & \underline{Da} sky ... &  86 f8 16 c5 dd 100 28 d1 91 8f 48 3c ff 3a a6 e2 b1 31 23 91 17 73 64 86 be 6b 81 ad 5e 10 67 56 &
hint & [147, 217, 2582, 2976, 1718, 1599, 27, 1083, 471, 1461, 1076, 2255, 
 2875, 2696, 2793, 1015, 1477, 1271, 2856, 221, 961, 2839, 1789, 1845, 1157, 622, 758, 882, 210, 1846, 3009, 410] \\

\noalign{\smallskip}\hline\noalign{\smallskip} 

\multirow{4}{*}{DMAC-2} &The sky announced a beautiful day... & 
6f ed d1 fb 2f cf 56 fc a9 5e c8 1d 90 ec f7 4a df 42 1a 1e 3b 16 
62 54 90 81 a2 a4 7e 3f 8d db & red & [147, 217, 2582, 2976, 1718, 1599, 27, 1083, 471, 1461, 1076, 2255, 
 2875, 2696, 2793, 1015, 1477, 1271, 2856, 221, 961, 2839, 1789, 1845, 1157, 622, 758, 882, 210, 1846, 3009, 410]

\\
 
 & The sky ... & 
59 9 c2 a b1 ba 88 1b 12 e7 f3 a 65 71 87 7b 25 c4 20 
57 38 6e 54 b0 b8 19 74 5b d8 33 46 d & r\underline{i}d & [147, 217, 2582, 2976, 1718, 1599, 27, 1083, 471, 1461, 1076, 2255, 
 2875, 2696, 2793, 1015, 1477, 1271, 2856, 221, 961, 2839, 1789, 1845, 1157, 622, 758, 882, 210, 1846, 3009, 410]

\\

 & The sky ... & 
d7 92 60 73 0 dd ef fa 6 3 f2 c6 9b 62 c6 58 e7 59 31 a5 2d 
5e 34 67 7d a9 95 30 86 12 9a e0 & red & [\underline{149}, \underline{221}, 2582, 2976, 1718, 1599, 27, 1083, 471, 1461, 1076, 2255, 
 2875, 2696, 2793, 1015, 1477, 1271, 2856, 221, 961, 2839, 1789, 1845, 1157, 622, 758, 882, 210, 1846, 3009, 410]

\\

 & \underline{Da} sky ... &  54 c7 8 d3 d9 fc c1 ed 57 18 9d 74 62 d2 5d 35 5a cd 15 3c b7 19 9a 3c 79 1d 4c 68 69 3b d8 6b &
red & [147, 217, 2582, 2976, 1718, 1599, 27, 1083, 471, 1461, 1076, 2255, 
 2875, 2696, 2793, 1015, 1477, 1271, 2856, 221, 961, 2839, 1789, 1845, 1157, 622, 758, 882, 210, 1846, 3009, 410]\\ 

\noalign{\smallskip}\hline
\end{tabular}
\end{table}}

\section{Timings}

The number of operations per bit of input depend on block length $N$ and on the parameter $n$ of the graph $D(n,Q)$.
The number of field operations in the system of equations~\ref{1}  is $2(n-1)$ (one step of the walk in the graph $D(n,Q)$ 
costs $2(n-1)$ field operations).

To process one block of input with DMAC-1 we need two additions and one multiplication as specified by operator $N_t(w)$ (to calculate the 
first coordinate of the neighbor), $\mod$ operation and $2(n-1)$ field operations. After the block is processed we process vector $S$.
 Therefore, the number of operations per bit of input for DMAC-1 is given by the formula
\[\dfrac{(2n+2)\cdot l(M)+(2n+2)\cdot s}{N\cdot l(M)}=\dfrac{2n+2}{N}\left(1+\dfrac{s}{l(M)}\right),\]
where $r$ is the length of $S$ and $l(M)$ is the number of blocks, as specified above.

To process one block of input with DMAC-2 we need to add vectors over number field $\mathbb F_q$ which require $n$ field additions.
It gives us the following formula for the number of operations per bit of input when DMAC-2 is used
\[\dfrac{(2n+2+n)l(M)+(2n+2+n)s}{N\cdot l(M)}=\dfrac{3n+2}{N}\left(1+\dfrac{s}{l(M)}\right).\]

If the password $S$ is short (not more than 10 characters) and the message, for which we want to calculate the tag, is long 
(which is true in general when MACs are used) then 
$\frac{r}{l(M)}$ is very small.
Notice that the change of number field doesn't increase much the number of operations per bit of input (except for the fact that resulting vector has to 
be divided $\mod q$). 

\begin{exa}
Lets consider data like in Example \ref{ex:2}. 
The  length of $M$ in bits is $2000\cdot l(q)=2000\cdot 8$. The block length is $N=32$ bits so the number of blocks 
$l(M)=\dfrac{2000\cdot8}{32}=500$.
In this case the number of field operations per bit of input is 
\[\dfrac{130}{32}\left(1+\dfrac{10}{500}\right)+\dfrac{64}{2000\cdot 8}\approx 4,\]
which is very efficient. 
\end{exa}

\section{Conclusions}

A new technique for message authentication was presented in this work. To the best of our knowledge,
the family of graphs $D(n,q)$ has never been used before in this context. 
The algorithms here introduced, for DMAC-1 and DMAC-2 respectively, were implemented in Python and tested with different inputs. 
The results of our tests and the theoretical base show that the technique we introduce is a very efficient and safe approach 
to compute message authentication code.

\subsection*{Acknowledgement}
The authors would like to express their gratitude to Vasyl Ustimenko for sharing his knowledge about graphs $D(n,q)$, which made 
this research possible. Special thanks also to Stanislaw Radziszowski for his useful remarks.

%
%


\begin{thebibliography}{30}



\bibitem{biggs88}
     Norman Biggs, Graphs with large girth,
    Ars Combinatoria, 25C (1987), 73--80 .

\bibitem{biggs98}
  Norman Biggs, Constructions for cubic graphs with large girth,
     The electronic jurnal of Combinatorics Vol. 5 (1998).
\bibitem{biggsboshier90}
    N. L. Biggs and A.G Boshier,
    Note on the girth of Ramanujan graphs, Journal of Combinatorial Theory,
		Vol. 49 (1990), 190--194.
  
\bibitem{biggshoare83}
    N. L. Biggs and M. J. Hoare,
    The sextet construction for cubic graphs,
    Combinatorica,
		Vol. 3 (1983),
		153--165.    
  

\bibitem{boll2004}
    B\'{e}la Bollob\'{a}s,
    Extremal Graph Theory,
Dover Publications, 2004.

\bibitem{cglhash}
Denis X. Charles, Eyal Z. Goren and Kristin E. Lauter, 
Cryptographic hash functions from expander graphs,
Journal of Cryptology,
Vol. 22
(2009), 93--113.


	
\bibitem{erd59}
    Erd{\H{o}}s, Paul,
    Graph Theory and Probability,
		Modern Birkhauser Classics,
		Classic Papers in Combinatorics (1987), 276--280.
\bibitem{furedi1995}
    Z. Furedi and F. Lazebnik and A. Seress and V.A. Ustimenko and A.J. Woldar,
    Graphs of Prescribed Girth and Bi-Degree,
    Journal of Combinatorial Theory, Series B,
		Vol. 64 (1995),
		228--239. 
  
\bibitem{gl1997}
    P. Guinand and J. Lodge, Tanner type codes arising from large girth graphs,
    Proceedings of Canadian Workshop on Information Theory CWIT ’97, Toronto, Ontario, Canada (1997),
		5--7 
  		
\bibitem{Goubin2011}
Goubin Louis, Patarin Jacques and Yang Bo-Yin, Multivariate Cryptography, Encyclopedia of Cryptography and Security, Springer US, 2011.


	
\bibitem{expander}
    S. Hoory and N. Linial and A. Wigderson,
    Expander graphs and their applications,
    Bulletin of the American Mathematical Society,
		Vol. 43 (2006),
		439--561. 
  
\bibitem{TSUD92}
G. Tsudik, Message authentication with one-way hash functions, Preecedings INFOCOM '92 (1992).
  
	
\bibitem{imrich84}
    Imrich, Vrto,
    Explicit construction of graphs without small cycles,
    Combinatorica,
		Vol. 4 (1984),
		53--59. 

	
\bibitem{klisowski2012}
    M. Klisowski and V. Ustimenko ,
    On the Comparison of Cryptographical Properties of Two Different Families of Graphs with Large Cycle Indicator,
    Mathematics in Computer  Science,
		Vol. 6 (2012),
		181--198.
  
  
\bibitem{usty1995}
    F. Lazebnik and V. A. Ustimenko,
    Explicit construction of graphs with an arbitrary large girth and of large size,
    Discrete Applied Mathematics,
		Vol. 60 (1995), 275--284.  
  
	
\bibitem{lazebnik1995}
    F. Lazebnik and V. A. Ustimenko and A. Woldar ,
    A New Series of Dense Graphs of High Girth,
    Bull (New Series) of AMS,
		Vol. 32 (1995),
		73--79.
  

\bibitem{laz99}
    F. Lazebnik and V. A. Ustimenko and A.J. Woldar,
    Polarities and $2k$-cycle-free graphs,
    Discrete Mathematics (1999),
		503--513. 
  
\bibitem{lps88}
     A. Lubotzky and R. Phillips and P. Sarnak,
    Ramanujan graphs,
    Combinatorica,
		Vol. 8 (1988),
		261--277.
  
	
\bibitem{marguli82}
    Grigorij A. Margulis,
    Explicit constructions of graphs without short cycles and low density codes,
    Combinatorica,
		Vol. 2 (1982),
		71--78.
  


\bibitem{marguli88}
       Grigorij A. Margulis, Explicit group-theoretical constructions of combinatorial schemes and their application to the design of 
   expanders and concentrators,
    Problems of Informations Transmission,
		Vol. 24 (1988),
		51--60. 
  
	
\bibitem{diehard}
    Marsaglia, Gorge, The Marsaglia Random Number CDROM, with The Diehard Battery of Tests of Randomness,
		produced at Florida State University under a grant from The National Science Foundation, 1985.
  
\bibitem{morg}
    M. Morgenstern , Existence and explicit constructions of $q + 1$-regular Ramanujan graphs for every prime power $q$,
    Journal of Combinatorial Theory, Series B, 
		Vol. 62 (1994), 44--62.
  
\bibitem{pizer}
A.K. Pizer, Ramanujan Graphs and Hecke Operators, Bulletin of the AMS, Vol. 23, No 1 (1990).
	
\bibitem{muav2013}
    M. Polak and U. Roma\'nczuk and V. Ustimenko and A. Wr\'oblewska , On the applications of Extremal Graph Theory to Coding Theory 
    and Cryptography,
    Electronic Notes in Discrete Mathematics,
		Vol. 43 (2013),
		329--342.
  

\bibitem{stallings}
William Stallings, Cryptography and Network Security: Principles and Practice,

 3rd,
 Pearson Education, 2002.
 
\bibitem{Tillich2008}
Jean-Pierre Tillich and  Gilles Z{\'e}mor, Collisions for the LPS Expander Graph Hash Function, 
Advances in Cryptology -- EUROCRYPT 2008: 27th Annual International Conference on the Theory and Applications of Cryptographic Techniques, 
Istanbul, Turkey, April 13-17, 2008. Proceedings"
 (2008), 254--269. 


\bibitem{zem1}
Jean-Pierre Tillich and  Gilles Z{\'e}mor, Hashing with $SL_2$, Advances in Cryptology, Crypto’94, Lecture Notes in Computer Science, 
Vol. 839 (1994).


	
\bibitem{ustimenko1998}
    Vasyl Ustimenko, Coordinatisation of Trees and their Quotients,
    Voronoj's  Impact on Modern Science,
		 Vol. 2 (1998),
		125 -- 152.
  
\bibitem{ustimenko2005}
    Vasyl Ustimenko, Maximality of affine group and hidden graph cryptosystems,
    J. Algebra Discrete Math. (2005),
		133--150.
  


\bibitem{Ustimenko2007}
    Vasyl Ustimenko,
    On the extremal graph theory for directed graphs and its cryptographical applications,
    In:  Shaska T.,  Huffman W.C.,  Joener D. and  Ustimenko V., Advances in Coding Theory and Cryptography, Series on 
    Coding and Cryptology,
	Vol. 3 (2007), 181--.


\bibitem{usty2005}
    Vasyl Ustimenko, On linguistic dynamical systems, families of graphs of large girth and cryptography,
    Zapiski Nauchnykh Seminarov POMI,
		Vol. 326 (2005),  214--234.
  
\bibitem{ustimenko2013}
    V. Ustimenko  and A. Wr\'oblewska, On some algebraic aspects of data security in cloud computing,
    Proceedings of International conference: Applications of Computer Algebra, Malaga, Vol. 32 (2013),  144--147.
  
\bibitem{aweiss1984}
    Alfred Weiss
    \emph{Girth of bipartite sextet graphs}, Combinatorica, Vol. 4 (1984), 241--245.
    
\bibitem{aw}
Aneta Wroblewska, On some properties of graph based public keys, Albanian J. Math. 2 , no. 3 (2008),  229--234.

\bibitem{zem2}
Gilles Z{\'e}mor, Hash functions and Cayley Graphs, Designs, Codes and Cryptography, Vol. 4 (1994), 381--394.


	
\bibitem{generator2017}
    E. Zhupa, M. Polak, N. Marina, Efficient pseudorandom number generator with large girth graph $D(n,q)$,  to appear (2017).

  
\end{thebibliography}
\end{document}